\documentclass[preprint,11pt]{elsarticle}
\makeatletter
\def\ps@pprintTitle{%
 \let\@oddhead\@empty
 \let\@evenhead\@empty
 \def\@oddfoot{\centerline{\thepage}}%
 \let\@evenfoot\@oddfoot}
\makeatother
\usepackage{amssymb,amsmath,amsthm}
\usepackage{epsfig}
\usepackage{color}
\usepackage{makeidx}
\usepackage{bbm}
\usepackage[top=1in, bottom=1in, inner=1in, outer=1in]{geometry}
\usepackage{setspace}

\newtheorem{theorem}{Theorem}
\newtheorem*{thma}{Theorem A}
\newtheorem*{thmb}{Theorem B}

\newtheorem*{propa}{Proposition A}
\newtheorem*{propb}{Proposition B}
\newtheorem{definition}{Definition}
\newtheorem{lemma}{Lemma}

\newtheorem{corollary}{Corollary}
\newtheorem*{remark}{Remark}
\newcommand*\xbar[1]{%
  \hbox{%
    \vbox{%
      \hrule height 0.5pt 
      \kern0.4ex
      \hbox{%
        \kern-0.15em
        \ensuremath{#1}%
        \kern-0.15em
      }%
    }%
  }%
}

\begin{document}
\begin{frontmatter}

\title{On the density of branching Brownian motion in subcritical balls}

\author{Mehmet \"{O}z}
\ead{mehmet.oz@ozyegin.edu.tr}
\ead[url]{https://faculty.ozyegin.edu.tr/mehmetoz/}

\address{Department of Natural and Mathematical Sciences, Faculty of Engineering, \"{O}zye\u{g}in University, Istanbul, Turkey}

\begin{abstract}
We study the density of the support of a dyadic $d$-dimensional branching Brownian motion (BBM) in subcritical balls in $\mathbb{R}^d$. Using elementary geometric arguments and an extension of a previous result on the probability of absence of the support of BBM in linearly moving balls of fixed size, we obtain sharp asymptotic results on the degree of density of the support of BBM in subcritical balls. As corollaries, we obtain almost sure results about the large-time behavior of $r(t)$-enlargement of the support of BBM when the shrinking radius $r(t)$ is decaying sufficiently slowly. As a by-product, we obtain the lower tail asymptotics for the mass of BBM falling in linearly moving balls of exponentially shrinking radius, which is of independent interest. 
\end{abstract}

\vspace{3mm}

\begin{keyword}
Branching Brownian motion \sep density \sep large deviations \sep neighborhood recurrence 
\vspace{3mm}
\MSC[2010] 60J80 \sep 60F10 \sep 92D25
\end{keyword}

\end{frontmatter}

\pagestyle{myheadings}
\markright{Density of BBM\hfill}

\section{Introduction}\label{intro}

The setting in this paper is a branching Brownian motion (BBM) evolving in $\mathbb{R}^d$. It is well-known that typically the \emph{mass}, i.e., number of particles, of a BBM grows exponentially with time. To be precise, if $N_t$ denotes the total mass of a strictly dyadic BBM at time $t$, and $\beta$ is the branching rate, then 
$$\underset{t\rightarrow\infty}{\lim}N_t\, e^{-\beta t}=M>0 \quad \text{a.s.}$$ 
meaning that the limit exists and is positive almost surely. It is also known that the \emph{speed} of a strictly dyadic BBM is $\sqrt{2\beta}$, which means that typically for large time the support of BBM at time $t$ is contained in $B(0,\sqrt{2\beta}(1+\varepsilon)t)$, where we use $B(x,r)$ to denote a ball of radius $r$ and center $x$, but not contained in $B(0,\sqrt{2\beta}(1-\varepsilon)t)$ for any $0<\varepsilon<1$. Let us call $B(0,\sqrt{2\beta}(1-\varepsilon)t)$ a subcritical ball (see Definition~\ref{subcritical}). Then, a natural question concerns the spatial distribution of mass at time $t$: how homogeneously are the exponentially many particles spread out over a subcritical ball? If they are spread out sufficiently homogeneously, then one may formulate this in terms of the density of the support of BBM, and obtain quantitative results on the degree of density of BBM. This work presents fine results on the distribution of particles of BBM at time $t$ for large $t$, and mainly aims at answering the question of how dense the support of BBM is in subcritical balls. 

We first extend a previous result \cite[Corollary 2]{O2018} on the probability of absence of BBM in moving balls of fixed radius  to moving balls of time-dependent radius; then using this result and elementary geometric arguments, we obtain a large deviation result on the asymptotic behavior of the density of the support of BBM in subcritical balls. As corollaries, we show that for a suitably decreasing function $r:\mathbb{R}_+\to\mathbb{R}_+$, the $r(t)$-enlargement of the support of BBM at time $t$ fills up the entire subcritical zone asymptotically as $t\rightarrow\infty$, and obtain almost sure results on its volume.

\subsection{Formulation of the problem}

Let $Z=(Z(t))_{t\geq 0}$ be a $d$-dimensional strictly dyadic BBM with constant branching rate $\beta>0$. Here, $t$ represents time, and strictly dyadic means that every time a particle branches, it gives exactly two offspring. The process starts with a single particle, which performs a Brownian motion in $\mathbb{R}^d$ for a random exponential time of parameter $\beta>0$, at which the particle dies and simultaneously gives birth to two offspring. Similarly, starting from the position where their parent dies, each offspring particle repeats the same procedure as their parent independently of others and of the parent, and the process evolves through time in this way. The Brownian motions and exponential lifetimes of particles are all independent from one another. For each $t\geq 0$, $Z(t)$ can be viewed as a discrete measure on $\mathbb{R}^d$. Let $P_x$ and $E_x$, respectively, denote the probability and corresponding expectation for $Z$ when the process starts with a single particle at position $x\in\mathbb{R}^d$, that is, when $Z(0)=\delta_x$, denoting the Dirac measure at $x$. When $Z(0)=\delta_0$, we simply use $P$ and $E$. For a Borel set $B\subseteq \mathbb{R}^d$ and $t\geq 0$, we write $Z_t(B)$ to denote the number of particles, i.e., the \textit{mass}, of $Z$ that fall inside $B$ at time $t$. We write $N_t:=Z_t(\mathbb{R}^d)$ for the total mass at time $t$. The \emph{range} of $Z$ up to time $t$, and the full range of $Z$, are defined respectively as 
\begin{equation} \label{range}
R(t)=\underset{0\leq s\leq t}{\bigcup} \text{supp}(Z(s)),\quad \quad R=\underset{t\geq 0}{\bigcup} R(t).
\end{equation}

By the classical result of McKean \cite{MK1975}, it is well-known that the \emph{speed} of strictly dyadic BBM in one dimension is equal to $\sqrt{2\beta}$, which was later generalized to higher dimensions by Engl\"ander and den Hollander \cite{E2003}. More precisely, we have the following result.
\begin{thma}[Speed of BBM; \cite{MK1975,E2003}] 
Let $Z$ be a strictly dyadic BBM in $\mathbb{R}^d$. For $t\geq 0$ define $M_t:=inf\{r>0:\emph{supp}(Z(t))\subseteq B(0,r)\}$ to be the radius of the minimal ball that contains the support of BBM at time $t$. Then, in any dimension, 
$$M_t/t \rightarrow \sqrt{2\beta} \quad \text{in probability} \quad \text{as} \quad t\rightarrow\infty.$$
\end{thma}
Note that $M_t$ quantifies the spatial spread of BBM at time $t$ so that $M_t/t$ is a measure of the speed of BBM. More sophisticated results on the speed of BBM, such as almost sure results and higher order sublinear corrections, exist in the literature (see for example \cite{B1978,K2005}). For our purposes, Theorem A suffices; it says that typically for large $t$ and any $\varepsilon>0$, at time $t$ there will be particles outside $B(0,\sqrt{2\beta}(1-\varepsilon)t)$ but no particles outside $B(0,(\sqrt{2\beta}(1+\varepsilon)t)$. Therefore, when we study the density of the support of BBM at time $t$, to obtain meaningful results, we consider the density within a \emph{subcritical ball}, which we define as follows.  
\begin{definition}[Subcritical ball]\label{subcritical} \label{def1}
We call $B=(B(0,\rho_t))_{t\geq 0}$ a \emph{subcritical ball} if there exists $0<\varepsilon<1$ and time $t_0$ such that $B(0,\rho_t)\subseteq B\left(0,\sqrt{2\beta}(1-\varepsilon)t\right)$ for all $t\geq t_0$.
\end{definition}
\begin{remark} We use the term \emph{subcritical ball} both in the sense of a time-dependent ball $B=(B(0,\rho_t))_{t\geq 0}$ as in Definition~\ref{def1}, and also simply as a snapshot taken of a time-dependent ball at a fixed large time $t$ as $B(0,\rho_t)$.
\end{remark}

The current work is motivated by, and can be viewed as an extension of the following previous result. For a Borel set $B$ and $x\in\mathbb{R}^d$, we write $B+x:=\{y+x:y\in B\}$ in the sense of sum of sets. 
\begin{thmb}[Asymptotic probability of no particle inside a moving ball; \cite{O2018}]\label{thma}
Let $0\leq\theta<1$ and $B$ be a fixed ball in $\mathbb{R}^d$. Let $\textbf{e}$ be the unit vector in the direction of the center of $B$ if $B$ is not centered at the origin; otherwise let $\textbf{e}$ be any unit vector. For $t\geq 0$ define $B_t=B+\theta\sqrt{2\beta}t\textbf{e}$. Then, 
\begin{equation} \underset{t\rightarrow\infty}{\lim}\:\frac{1}{t}\log P\left(Z_t(B_t)=0\right)=-2\beta\left(\sqrt{2}-1\right)(1-\theta).
\end{equation}
\end{thmb} 
Theorem B gives the asymptotic behavior of the probability of absence of $Z$ in linearly moving balls of fixed size. For fine results on the distribution of particles of $Z$ in $\mathbb{R}^d$, we first extend Theorem B to linearly moving balls of time-dependent (suitably decreasing) radius (see Theorem~\ref{theorem1}). Then, via a covering by sufficiently many of such smaller balls, we obtain a large-deviation result on the degree of density of the support of BBM in subcritical balls (see Theorem~\ref{theorem2}). Henceforth, by an abuse of terminology, we will refer to the density of the support of BBM as the \emph{density of BBM}.

\subsection{History and related problems}

At the root of the present work is the strong law of large numbers (SLLN) for the local mass of BBM \cite[Corollary, p.\ 222]{W1967}, where Watanabe established an almost sure result on the asymptotic behavior of certain branching Markov processes, which covers the SLLN for local mass of BBM in fixed Borel sets in $\mathbb{R}^d$ as a special case. This was extended by Biggins \cite[Corollary 4]{B1992} to linearly moving Borel sets. The result of Biggins was originally cast in the setting of a branching random walk in discrete time, and extended in the same paper to the continuous setting of a BBM.

We now review various large-deviation (LD) results concerning the mass of BBM. First, we consider probabilities of absence or presence. Let $X_{\text{max}}(t)$ denote the position of the rightmost particle at time $t$ of a BBM in $\mathbb{R}$, and for any $d\geq 1$ let 
$$M_t:=\inf\{r>0:\text{supp}(Z(t))\subseteq B(0,r)\}$$
as in Theorem A. Set $v:=\sqrt{2\beta}$. Recall that by Theorem A, for large $t$, typically there are particles outside $B(0,rt)$ when $r<v$, but no particles outside $B(0,rt)$ when $r>v$. In \cite{CR1988}, the large-time asymptotics of LD probabilities $P(X_{\text{max}}(t)\geq rt)$ for $r>v$ were found in $d=1$, where $P(X_{\text{max}}(t)\geq rt)$ is a probability of presence in a region where there would typically be no particles. In \cite{E2004}, the asymptotics of LD probabilities $P(M_t\leq rt)$ for $0<r<v$ were found in any dimension, and note that in this case $P(M_t\leq rt)$ is a probability of absence in the region $\mathbb{R}^d\setminus B(0,rt)$ where there would typically be particles. Recently in \cite{D2017}, the asymptotics of $P(X_{\text{max}}(t)\leq rt)$ for $r<v$ were found when $d=1$, where $r$ was allowed to be negative as well. More generally, concerning the mass of BBM in time-dependent domains, fewer results are available. In \cite{A2017}, the upper tail asymptotics for the mass inside $[rt,\infty)$, $r<v$ were found for a BBM in $\mathbb{R}$. Due to \cite[Corollary 4]{B1992}, the mass inside $[rt,\infty)$ at time $t$ is typically $\exp[\beta(1-\theta^2)+o(t)]$, and in \cite{A2017}, LD probabilities $P(Z_t([rt,\infty))\geq e^{\beta at})$ were studied for $1-\theta^2<a<1$. 

The current work can be regarded as a follow-up to \cite{O2018}. Let $\textbf{e}$ be any unit vector in $\mathbb{R}^d$ and $r>0$ be fixed, and for $t>0$ define
$$B_t:=B(\theta vt \textbf{e},r),\quad \mathcal{B}_t:=B(0,\theta vt).$$ 
For $0<\theta<1$, the mass inside $B_t$ and the mass outside $\mathcal{B}_t$ both typically grow as $\exp[\beta(1-\theta^2)t+o(t)]$ for large time. In \cite[Thm.\ 1]{O2018} and \cite[Thm.\ 2]{O2018}, respectively, the asymptotic behavior of LD probabilities in the downward direction, $P(Z_t(B_t)<e^{\beta at})$ and $P(Z_t(\mathbb{R}^d\setminus \mathcal{B}_t)<e^{\beta at})$, were studied for $0\leq a<1-\theta^2$, where $a$ is an atypically small exponent for the growth of mass in the respective time-dependent domains. Note that Theorem B is a special case of \cite[Thm.\ 1]{O2018} where $a=0$. 
 
As for the density of BBM, in \cite{GK2003}, Grigor'yan and Kelbert established sufficient conditions for the transience and recurrence of a general class of BBMs with time-dependent branching rates and mechanisms on Riemannian manifolds, where the term \emph{recurrence} therein is equivalent to the almost sure density of the full range of BBM in the manifold.

\smallskip

We conclude this section with some often used terminology and the outline of the paper.
\begin{definition}[SES]
A generic function $g: \mathbb R_+\to \mathbb R$ is called \emph{super-exponentially small (SES)} if $\lim_{t\to\infty}\log g(t)/t=-\infty$.
\end{definition}

\begin{definition}[Overwhelming probability]
Let $(A_t)_{t>0}$ be a family of events indexed by time $t$. We say that $A_t$ occurs \emph{with overwhelming probability} as $t\rightarrow\infty$ if there is a constant $c>0$ and time $t_0$ such that 
$$P(A_t^c)\leq e^{-ct} \quad \text{for all} \quad t\geq t_0, $$ 
where $A^c$ denotes the complement of event $A$.
\end{definition}

{ \bf Outline:} The rest of the paper is organized as follows. In Section~\ref{section2}, we present our main results. In Section~\ref{section3}, we develop the preparation needed, including the statement and proof of several introductory results, for the proofs of Theorem~\ref{theorem1} and Theorem~\ref{theorem2}. Section~\ref{section4} is on the large deviations of the mass of BBM in moving and shrinking balls, including the proof of Theorem~\ref{theorem1}. Section~\ref{section5} is on the density of BBM in subcritical balls, including the proof of Theorem~\ref{theorem2}. In Section~\ref{section6}, we prove almost sure results on the large-time behavior of $r(t)$-enlargement of the support of BBM when the shrinking radius $r(t)$ is exponentially small in $t$.

\section{Results}\label{section2}

Our first result is a large deviation result, giving the large-time asymptotic rate of decay for the probability that the mass of BBM inside a linearly moving and exponentially shrinking ball is atypically small on an exponential scale. It is an extension of \cite[Thm.\ 1]{O2018}, where linearly moving balls of fixed size were considered. Here, the radius of the moving ball is time-dependent as well.

\begin{theorem}[Lower tail asymptotics for mass inside a moving and shrinking ball] \label{theorem1}
Let $0\leq\theta<1$, $0\leq k<(1-\theta^2)/d$, $r_0>0$ and $\textbf{e}$ be any unit vector in $\mathbb{R}^d$. Let $x:\mathbb{R}_+\rightarrow\mathbb{R}_+$ and $r:\mathbb{R}_+\rightarrow\mathbb{R}_+$ be defined by $x(t)=\theta\sqrt{2\beta}t$ and $r(t)=r_0\,e^{-\beta kt}$. For $t\geq 0$ define $B_t=B(x(t)\textbf{e},r(t))$. Then, for $0\leq a<1-\theta^2-kd$,
\begin{equation} \underset{t\rightarrow\infty}{\lim}\frac{1}{t}\log P\left(Z_t(B_t)<e^{\beta a t}\right)=-\beta \times I(\theta,k,a), \label{ld0}
\end{equation}
where
\begin{equation}
I(\theta,k,a)=\underset{\rho\in(0,\bar\rho]}{\inf}\left[\rho+\frac{\left(\sqrt{(1-\rho)^2-(a+kd)(1-\rho)}-\theta\right)^2}{\rho}\right], \label{ld}
\end{equation}
and
\begin{equation}\bar\rho=\bar\rho(\theta,k,a)=1-\frac{a+kd}{2}-\sqrt{\left(\frac{a+kd}{2}\right)^2+\theta^2}. \label{rhobar}
\end{equation}
\end{theorem}

\begin{remark} In terms of the BBM's optimal strategies for realizing the LD event $\left\{Z_t(B_t)< e^{\beta a t}\right\}$, this means (see the proof of Theorem~\ref{theorem1} for details) to realize $\left\{Z_t(B)< e^{\beta a t}\right\}$: \newline
the system suppresses the branching completely, and sends the single particle to a distance of $\sqrt{2\beta}(\sqrt{(1-\hat{\rho})^2-(a+kd)(1-\hat{\rho})}-\theta)t+o(t)$ in the opposite direction of the center of $B_t$ over $[0,\hat{\rho}t]$, and then behaves `normally' in the remaining interval $[\hat{\rho}t,t]$,  
where $\hat\rho$ denotes the unique minimizer of the optimization problem in \eqref{ld}. 	
\end{remark}

\begin{remark}Theorem~\ref{theorem1} implies in particular that as the dimension $d$ increases it becomes easier on a logarithmic scale to send exponentially few particles to $B_t$ at time $t$.
\end{remark}

The optimization problem in \eqref{ld} is identical to the one in \cite[Eq.\ 4]{O2018} with the replacement of the parameter $a$ therein by $a+kd$. The following can be shown to hold: 
\begin{enumerate}
	\item[(i)] The function to be minimized in \eqref{ld}, call $f$, is strictly convex, and has a unique minimizer on $(0,1-a-kd)$. Denote this minimizer by $\hat\rho=\hat\rho(\theta,k,a)$. Then, $\hat\rho$ satisfies $\hat\rho\leq\bar\rho$.
	\item[(ii)] If we consider $f$ as $f_{\theta,k,a}$, and keep any two of the three parameters $\theta,k,a$ fixed, both $\hat\rho$ and $f(\hat\rho)$ are strictly decreasing in the remaining parameter over the allowed set of values for that parameter. This is intuitively obvious since it becomes easier to send less than $e^{\beta at}$ particles to $B_t$, i.e., the event $\{Z_t(B_t)<e^{\beta a t}\}$ becomes more likely, as either of $\theta,k,a$ increases.
\end{enumerate}
For the proofs of these results and more details on the optimization problem in \eqref{ld}, we refer the reader to \cite[Sect.\ 5]{O2018}.

Theorem~\ref{theorem1} leads to the following almost sure result concerning the mass of BBM inside moving and shrinking balls.

\begin{corollary}[Almost sure growth inside a moving and shrinking ball] \label{corollary1}
Let $0\leq\theta<1$, $0\leq k<(1-\theta^2)/d$, $r_0>0$ and $\textbf{e}$ be any unit vector in $\mathbb{R}^d$. Let $x:\mathbb{R}_+\rightarrow\mathbb{R}_+$ and $r:\mathbb{R}_+\rightarrow\mathbb{R}_+$ be defined by $x(t)=\theta\sqrt{2\beta}t$ and $r(t)=r_0\,e^{-\beta kt}$. For $t\geq 0$ define $B_t=B(x(t)\textbf{e},r(t))$. Then,
\begin{equation}\underset{t\rightarrow\infty}{\lim}\frac{1}{t}\log Z_t(B_t)=\beta(1-\theta^2-kd) \quad \text{a.s.} \label{asresult1}
\end{equation}
\end{corollary} 

\begin{remark}Corollary~\ref{corollary1} can be viewed as an extension of \cite[Corollary 4]{B1992} to linearly moving balls of time-dependent radius. The exponential growth rate of $Z_t(B_t)$ consists of three pieces: the first term on the right-hand side of \eqref{asresult1} contributes positively and is simply the growth rate of the global mass of BBM, the second and third terms contribute negatively to the exponent, and come from a `one-particle picture,' where a Brownian particle has linear displacement and falls inside a specified ball of exponentially decaying radius.
\end{remark}

Next, we present the main result of this work, which is on the density of BBM in subcritical balls. First, we recall the following standard definition.

\begin{definition} A set $S$ is said to be \emph{$\delta$-dense} in $X\subseteq\mathbb{R}^d$ for a given $\delta>0$ if for any $x$ in $X$, there exists $s$ in $S$ such that $|s-x|<\delta$.
\end{definition}

\begin{theorem}[LD on density of BBM] \label{theorem2}
Let $0<\theta<1$, $0\leq k<(1-\theta^2)/d$, and for $t>0$ define $\rho_t:=\theta\sqrt{2\beta}t$. For $t>0$ and a function $r:\mathbb{R}_+\to\mathbb{R}_+$, define the event $A_t^r$ as 
$$A_t^r:=\left\{\text{supp}(Z(t))\:\: \text{is not $r(t)$-dense in $B(0,\rho_t)$}\right\}.$$
If $r$ is defined by $r(t)=r_0\,e^{-\beta kt}$, where $r_0>0$, then
\begin{equation} \underset{t\rightarrow\infty}{\lim}\frac{1}{t}\log P\left(A_t^r\right)=-\beta \times I(\theta,k,0). \label{thm2}
\end{equation}
\end{theorem}

Note that the rate constant in \eqref{thm2} is a measure of how fast the support of BBM becomes $r(t)$-dense in the linearly expanding subcritical ball $B=(B(0,\rho_t))_{t\geq 0}$.

Via a Borel-Cantelli argument, Theorem~\ref{theorem2} leads to the following corollary, which is on the density of the full range of BBM. We provide a proof for completeness.

\begin{corollary}[Density of BBM]\label{corollary2}
Let $Z$ be a strictly dyadic BBM with constant branching rate $\beta>0$, and let $R$ denote the full range of $Z$ as defined in \eqref{range}. Then, in any dimension $d\geq 1$, $R$ is dense in $\mathbb{R}^d$ almost surely.
\end{corollary}

\begin{proof}
For concreteness, set $\theta=1/\sqrt{2}$ in the definition of $\rho_t$ in the statement of Theorem~\ref{theorem2} so that $\rho_t=\sqrt{\beta}t$. For $n\in\mathbb{N}$, let $F_n$ be the event that $R(n)$ is not $(1/n)$-dense in $B(0,\rho_n)$. Note that for any $k$, $1/n \geq e^{-kn}$ for all large $n$, and for any $n$, $\text{supp}(Z(n))\subseteq R(n)$. Therefore, Theorem~\ref{theorem2} implies that there exist $c>0$ and $j\in\mathbb{N}$ such that for $n\geq j$, $P(F_n)\leq e^{-cn}$. Since $\sum_{n=j}^\infty P(F_n)\leq 1/(1-e^{-c})<\infty$, by Borel-Cantelli lemma, with probability one, only finitely many $F_n$ occur. This means that $P(\Omega_0)=1$, where 
$$\Omega_0:=\left\{\omega: \exists \: n_0=n_0(\omega) \: \text{such that} \: \forall \: n\geq n_0, \: R(n)(\omega) \: \text{is $(1/n)$-dense in} \: B(0,\rho_n)\right\}.$$
Let $\omega\in\Omega_0\,$. Then $\exists \: n_0(\omega)$ such that $\forall \: n\geq n_0$, $R(n)(\omega)$ is $(1/n)$-dense in $B(0,\rho_n)$. Let $x\in\mathbb{R}^d$ and $\varepsilon>0$. Consider $B(x,\varepsilon)$. Choose $N$ large enough so that 
\begin{equation} N>n_0,\quad x\in B(0,\rho_N), \quad \frac{1}{N}<\varepsilon. 
\end{equation}
For instance, choosing $N>\max\left\{n_0,\frac{|x|}{\sqrt{\beta}},\frac{1}{\varepsilon}\right\}$ suffices. Then, $B(x,\varepsilon)\cap R(N)\neq \emptyset$, which in view of $R(N)\subseteq R$ implies that $B(x,\varepsilon)\cap R\neq \emptyset$. Therefore, $P(R\:\text{is dense in}\:\mathbb{R}^d)\geq P(\Omega_0)=1$. 
\end{proof}

\begin{remark}
We note that Corollary~\ref{corollary2} is not a new result. Via a similar Borel-Cantelli argument as the one above, one can deduce the almost sure density of the full range of BBM from Watanabe's SLLN \cite[Corollary, p.\ 222]{W1967} for the local mass of BBM. Also, Corollary~\ref{corollary2} can be recovered as a special case of \cite[Thm.\ 8.1]{GK2003}, which provides sufficient conditions for the transience or recurrence of a general class of branching diffusions on Riemannian manifolds, including the BBM in $\mathbb{R}^d$.
\end{remark}

\smallskip

The concept of $r$-density of $Z(t)$ naturally leads to the following definition.
\begin{definition}[Enlargement of BBM]
Let $Z=(Z(t))_{t\geq 0}$ be a BBM. For $t\geq 0$, we define the \emph{$r$-enlargement of BBM at time $t$} corresponding to $Z$ as
$$Z_t^r:=\bigcup_{x\,\in\,\emph{supp}(Z(t))}B(x,r).$$
\end{definition}

For a (typically non-increasing) function $r:\mathbb{R}_+\rightarrow\mathbb{R}_+$, we may similarly define the \emph{$r(t)$-enlargement of BBM} as
$$Z_t^{r_t}:=\bigcup_{x\,\in\,\emph{supp}(Z(t))}B(x,r_t) ,$$
where we have set $r_t=r(t)$ for notational convenience. 

The following results concern the large-time asymptotic behavior of the $r_t$-enlargement of BBM in $\mathbb{R}^d$. As a corollary of Theorem~\ref{theorem2}, we first state that, with probability one, an $r_t$-enlargement of BBM covers the corresponding subcritical ball $B(0,\rho_t)$ eventually for an exponentially decaying $r$ provided that the decay rate satisfies the condition in Theorem~\ref{theorem2}. Then, we give an almost sure result on the large-time asymptotic behavior of the volume of $r_t$-enlargement of BBM. 

Throughout the manuscript, for a Borel set $A\subseteq \mathbb{R}^d$, we say \emph{volume} of $A$ to refer to its Lebesgue measure, which we denote by $\textsf{vol}(A)$, and use $\omega_d$ to denote the volume of the $d$-dimensional unit ball.  

\begin{corollary}[Almost sure density of BBM] \label{corollary3}
Let $0<\theta<1$, $0\leq k<(1-\theta^2)/d$, $r_0>0$ and $r:\mathbb{R}_+\rightarrow\mathbb{R}_+$ be defined by $r(t)=r_0\,e^{-\beta kt}$. For $t>0$ define $\rho_t:=\theta\sqrt{2\beta}t$. Then,
$$P(\Omega_0)=1,\:\:\text{where}\quad\Omega_0:=\left\{\omega: \exists \: t_0=t_0(\omega) \: \text{such that} \:\: \forall \: t\geq t_0, \: B(0,\rho_t)\subseteq Z^{r_t}_t(\omega)\right\}.$$
\end{corollary}

\begin{theorem}[Almost sure growth of enlargement of BBM]\label{theorem3}
Let $0\leq k\leq 1/d$, $r_0>0$ and $r:\mathbb{R}_+\rightarrow\mathbb{R}_+$ be defined by $r(t)=r_0\,e^{-\beta kt}$. Then,
\begin{equation}\underset{t\rightarrow\infty}{\lim}\frac{\textsf{vol}\left(Z^{r_t}_t\right)}{t^d}=[2\beta(1-kd)]^{d/2}\omega_d\ \quad \text{a.s.} \label{asresult2}
\end{equation}
\end{theorem}

\section{Preparations}\label{section3}

{ \bf Notation:} We introduce further notation for the rest of the manuscript. For $x\in\mathbb{R}^d$, we use $|x|$ to denote its Euclidean norm. We use $c, c_0, c_1,\ldots$ as generic positive constants, whose values may change from line to line. If we wish to emphasize the dependence of $c$ on a parameter $p$, then we write $c_p$ or $c(p)$. We use $\mathbb{R}_+$ to denote the set of nonnegative real numbers, and write $o(t)$ to refer to $g(t)$, where $g:\mathbb{R}_+\to\mathbb{R}$ is a generic function satisfying $g(t)/t\rightarrow 0$ as $t\rightarrow\infty$.

We denote by $X=(X(t))_{t\geq 0}$ a generic standard Brownian motion in $d$-dimensions, and use $\mathbf{P}_x$ and $\mathbf{E}_x$, respectively, as the law of $X$ started at position $x\in\mathbb{R}^d$, and the corresponding expectation. Also, for $t>0$, $x,y\in\mathbb{R}^d$, and a Borel set $A\subseteq\mathbb{R}^d$, we denote by $p(t,x,y)$ and $p(t,x,A)$, respectively, the Brownian transition kernel and the probability that a Brownian motion that starts at $x$ falls inside $A$ at time $t$. We set $p(t,A):=p(t,0,A)$.

\smallskip

The following result says that the probability that there are no particles of BBM in a ball of fixed radius is an increasing function of the distance between the center of the ball and the starting point of the BBM. This is intuitively obvious, and is a direct consequence of the facts that the Brownian transition kernel is a decreasing function of $|x-y|$ and that each particle of BBM performs an independent Brownian motion while alive.

\begin{lemma}[Monotonicity of probability of absence] \label{lemma1}
Let $x_1$ and $x_2$ be in $\mathbb{R}^d$ with $|x_1|>|x_2|$, and $r>0$. Define $B_1:=B(x_1,r)$ and $B_2:=B(x_2,r)$. Then for any $t>0$,
$$P\left(Z_t(B_1)=0\right)\geq P\left(Z_t(B_2)=0\right) .$$
\end{lemma}

\begin{proof} Fix $r>0$ and let $g:\mathbb{R}_+\times\mathbb{R}^d\to[0,1]$ be defined by $g(t,x)=P\left(Z_t(B(x,r))=0\right)$. Condition on the first branching time as
\begin{align}
g(t,x)&=e^{-\beta t}[1-p(t,B(x,r))]+\int_0^t \mathbf{E}_0\left[g^2(t-s,x-X_s)\right]\beta e^{-\beta s}ds \nonumber \\
&=e^{-\beta t}[1-p(t,B(x,r))]+\int_0^t \mathbf{E}_0\left[g^2(u,x-X_{t-u})\right]\beta e^{-\beta (t-u)}du.
\end{align}
Then,
\begin{align}
g(t,x_2)-g(t,x_1)&=e^{-\beta t}[p(t,B_1)-p(t,B_2)]+ \nonumber \\
&\:\:\:\int_0^t\mathbf{E}_0\left[g^2(u,x_2-X_{t-u})-g^2(u,x_1-X_{t-u})\right]\beta e^{-\beta (t-u)}du \nonumber \\
&\leq \int_0^t\mathbf{E}_0\left[g^2(u,x_2-X_{t-u})-g^2(u,x_1-X_{t-u})\right]\beta e^{-\beta (t-u)}du, \label{eq30}
\end{align} 
where we have used that in the first line the first term on the right-hand side is negative due to the monotonicity of $p(t,x,y)$ in $|x-y|$. Define 
$$w(t,x):=g(t,x_2-x)-g(t,x_1-x),\quad \overline{w}:=w\vee 0,$$ 
where we use $a\vee b$ to denote the maximum of the numbers $a$ and $b$. Note that
\begin{align}
g^2(u,x_2-x)-g^2(u,x_1-x)&=[g(u,x_2-x)+g(u,x_1-x)][g(u,x_2-x)-g(u,x_1-x)] \nonumber \\
&\leq 2\,\overline{w}(u,x). \nonumber
\end{align}
It follows from \eqref{eq30} that
\begin{equation} \overline{w}(t,0)\leq \int_0^t \mathbf{E}_0\left[2\,\overline{w}(u,X_{t-u})\right]\beta e^{-\beta (t-u)}du. \label{eq31}
\end{equation}
Note that if $\overline{w}(t,0)=0$, then \eqref{eq31} holds since the right-hand side is nonnegative, and if $\overline{w}(t,0)>0$, then \eqref{eq31} holds by definition of $\overline{w}$ and by \eqref{eq30}. For $0\leq u\leq t$, define $F(u):=\mathbf{E}_0\left[\overline{w}(u,X_{t-u})\right]$, and note that $F(t)=\overline{w}(t,0)$. Then, \eqref{eq31} yields 
\begin{equation} F(t)\leq \int_0^t 2\beta e^{-\beta(t-u)}F(u)du,
\end{equation}
and by Gr\"onwall's inequality we conclude that $F(t)\leq 0$. Hence, $\overline{w}(t,0)\leq 0$. But $\overline{w}(t,0)\geq 0$ by definition. Therefore, $\overline{w}(t,0)=0$, that is, $g(t,x_2)-g(t,x_1)\leq 0$, which means that $g(t,x_1)\geq g(t,x_2)$ as claimed.
\end{proof}

Next, we list two well-known results; the first one is about the global growth of branching systems, and the second one about the large-time asymptotic probability of atypically large Brownian displacements. These results will be useful in the proofs of the main theorems and Lemma~\ref{lemma1}. For the proofs of Proposition A and Proposition B, see for example \cite[Sect.\ 8.11]{KT1975} and \cite[Lemma 5]{OCE2017}, respectively.

\begin{propa}[Distribution of mass in branching systems]\label{proposition2}
For a strictly dyadic continuous-time branching process $N=(N(t))_{t\geq 0}$ with constant branching rate $\beta>0$, the probability distribution at time $t$ is given by 
\begin{equation} P(N(t)=k)=e^{-\beta t}(1-e^{-\beta t})^{k-1},\quad k\geq 1, \nonumber
\end{equation}
from which it follows that
\begin{equation} P(N(t)>k)=(1-e^{-\beta t})^k  \label{eqprop}.
\end{equation}
\end{propa}

\begin{propb}[Linear Brownian displacements]\label{proposition3}
Let $X=(X(t))_{t\geq 0}$ represent a standard $d$-dimensional Brownian motion starting at the origin, and $\mathbf{P}_0$ the corresponding probability. Then, for $\gamma>0$ as $t\rightarrow\infty$,
\begin{equation}  \mathbf{P}_0\left(\underset{0\leq s\leq t}{\sup}|X(s)|>\gamma t\right)=\exp[-\gamma^2t/2+o(t)].  \label{prop3}
\end{equation}
\end{propb}

\section{Mass in a moving and shrinking ball}\label{section4}

The following lemma says that exponentially few particles in a moving and shrinking ball, is exponentially unlikely. It constitutes the first step of a two-step bootstrap argument, which we use to prove the upper bound of \eqref{ld0} in Theorem~\ref{theorem1}. The proof of the upper bound of Theorem~\ref{theorem1} will sharpen the constant on the right-hand side of \eqref{eq40} below.

\begin{lemma}\label{lemma2}
Let $0\leq \theta<1$, $0\leq k\leq (1-\theta^2)/d$, $r_0>0$, and $\textbf{e}$ be any unit vector in $\mathbb{R}^d$. Let $x:\mathbb{R}_+\to\mathbb{R}_+$ and $r:\mathbb{R}_+\to\mathbb{R}_+$ be defined by $x(t)=\theta\sqrt{2\beta}t$ and $r(t)=r_0 e^{-\beta kt}$. For $t\geq 0$ define $B_t=B(x(t)\textbf{e},r(t))$. Then, for each $0\leq a<1-\theta^2-kd$, there exists a constant $c=c(\beta,d,\theta,k,a)>0$ such that
\begin{equation} \underset{t\rightarrow\infty}{\limsup}\:\frac{1}{t}\log P\left(Z_t(B_t)<e^{\beta a t}\right)\leq-c. \label{eq40}
\end{equation} 
\end{lemma}

\begin{remark}
Note that $B=(B_t)_{t\geq 0}$ represents a linearly moving and exponentially shrinking ball. Using a many-to-one formula, we have
\begin{align}
E[Z_t(B_t)]=E[Z_t(\mathbb{R}^d)]\times p(t,B_t)&=e^{\beta t}\times\frac{1}{(2\pi t)^{d/2}}\int_{B_t}e^{-|x|^2/(2t)}\text{d}x \nonumber \\
&=e^{\beta t(1-\theta^2-kd)+o(t)}, \label{panter2}
\end{align}
where $p(t,A)$ is as before the Brownian transition probability from the origin to the Borel set $A$ at time $t$. Since $a<1-\theta^2-kd$ in the lemma above, $a$ is an atypically small exponent for the mass in $B_t$ at time $t$.
\end{remark}

\begin{proof}To start the proof, for $0\leq a<1-\theta^2-kd$ and $t>0$, let
$$A_t:=\left\{Z_t(B_t)<e^{\beta at}\right\},$$
and split the interval $[0,t]$ into two pieces as $[0,\delta t]$ and $[\delta t,t]$, where $0<\delta<1$ is small enough so as to satisfy $$a<1-\theta^2-kd-\delta. $$ 

For $t\geq 0$, set $x_t=x(t)$ and $r_t=r(t)$ for notational convenience. Consider the ball $B(x_t\textbf{e},r_0)$ so that $B_t\subseteq B(x_t\textbf{e},r_0)$ for all $t>0$. Next, for $t>0$, define the event
$$E_t:=\left\{Z_{t-1}\left(B(x_t\textbf{e},r_0)\right)\geq \exp\left[\beta(1-\theta^2-\delta)t\right]\right\},$$
and estimate
\begin{equation}
P(A_t)\leq P(A_t\mid E_t)+P(E_t^c). \label{eq41}
\end{equation}
Using \cite[Theorem 1]{O2018}, since $\beta(1-\theta^2-\delta)$ is an atypically small exponent for the mass inside $B(x_t\textbf{e},r_0)$ at time $t-1$, for all large $t$, $P(E_t^c)$ can be bounded from above as
\begin{equation}
P(E_t^c)\leq e^{-c_1t} \label{eq42}
\end{equation}
for some $c_1=c_1(\beta,\theta,\delta)>0$. (Note that $\delta=\delta(d,\theta,k,a)$.) Next, we show that $P(A_t\mid E_t)$ on the right-hand side of \eqref{eq41} is SES in $t$. 

Conditional on the event $E_t$, there are at least $\exp\left[\beta(1-\theta^2-\delta)t\right]$ particles in $B(x_t\textbf{e},r_0)$ at time $t-1$. Apply the branching Markov property at time $t-1$. For an upper bound on the mass inside $B_t$ at time $t$, neglect possible branching of the particles present in $B(x_t\textbf{e},r_0)$ at time $t-1$ over the period $[t-1,t]$, and assume that each one evolves as an independent Brownian particle over $[t-1,t]$ starting from her position at time $t-1$. Uniformly over $x\in B(0,r_0)$, a standard calculation yields
\begin{align}
p(1,x,B(0,r_t))=&\int_{B(0,r_t)}p(1,x,y)dy=\frac{1}{(\sqrt{2\pi})^d}\int_{B(0,r_t)}e^{-|y-x|^2/2}dy \nonumber \\
\geq&\, \frac{e^{-(r_0+r_t)^2/2}}{(\sqrt{2\pi})^d}\,\textsf{vol}\left(B(0,r_t)\right) \nonumber \\
\geq&\, \frac{e^{-(2 r_0)^2/2}}{(\sqrt{2\pi})^d}\,\omega_d r_t^d\,=\,c_2 \exp[-\beta(kd)t] \label{eq43}
\end{align}
for some constant $c_2>0$, and we have used in the second inequality that $r_t\leq r_0$ for all $t>0$. By translation invariance, uniformly over $x\in B(x_t\textbf{e},r_0)$,
$$p(1,x,B(x_t\textbf{e},r_t))\geq c_2 \exp[-\beta(kd)t].$$
Now for $t>t_0$, where $t_0$ is large enough, let 
$$p_t:=c_2 e^{-\beta(kd)t},\:\: q_t:=1-p_t,\:\: M_t:=\left\lceil e^{\beta(1-\theta^2-\delta)t}\right\rceil,$$ and let $Y_t$ be a random variable, which under the law $Q$, has a binomial distribution with parameters $M_t$ and $p_t$. (Here, $p_t$ is the probability of `success,' and $M_t$ is the number of trials.) Note that each particle present in $B(x_t\textbf{e},r_0)$ at time $t-1$ moves independently of others over $[t-1,t]$, and that conditional on $E_t$ there are at least $M_t$ particles in $B(x_t\textbf{e},r_0)$ at time $t-1$. Therefore, it follows that 
\begin{equation}
P(A_t\mid E_t)\leq Q(Y_t\leq e^{\beta at}). \label{panter}
\end{equation}
We now bound $Q(Y_t\leq e^{\beta at})$ from above as 
\begin{equation}
Q(Y_t\leq e^{\beta at})\leq \sum_{k=0}^{\left\lceil e^{\beta at}\right\rceil}Q(Y=k)=\sum_{k=0}^{\left\lceil e^{\beta at}\right\rceil}\binom{M_t}{k}p_t^k q_t^{M_t-k} 
\leq\sum_{k=0}^{\left\lceil e^{\beta at}\right\rceil}(M_t)^{\left\lceil e^{\beta at}\right\rceil}q_t^{M_t}, \label{eq44}
\end{equation}
where we have used that $p_t\leq q_t$ for all large $t$, and $\binom{M_t}{k}\leq (M_t)^{\left\lceil e^{\beta at}\right\rceil}$ for $0\leq k\leq\left\lceil e^{\beta at}\right\rceil$. We then bound $q_t^{M_t}$ from above as
\begin{align}
q_t^{M_t}=\left(1-c_2 e^{-\beta(kd)t}\right)^{\left\lceil \exp[\beta(1-\theta^2-\delta)t]\right\rceil}\leq &
\left[\left(1-\frac{c_2}{e^{\beta(kd)t}}\right)^{e^{\beta(kd)t}}\right]^{\frac{e^{\beta(1-\theta^2-\delta)t}}{e^{\beta(kd)t}}} \nonumber \\
\leq & \exp\left[-c_2\,\frac{e^{\beta(1-\theta^2-\delta)t}}{e^{\beta(kd)t}}\right]=\exp\left[-c_2 e^{\beta t(1-\theta^2-kd-\delta)}\right], \label{eq45}
\end{align}
where we have used the elementary estimate $(1+x)\leq e^x$ in passing to the second inequality. From \eqref{eq44} and \eqref{eq45}, it follows that for all large $t$
$$Q(Y_t\leq e^{\beta at})\leq \left(\left\lceil e^{\beta at}\right\rceil+1\right)\exp\left[2\beta(1-\theta^2)t e^{\beta at}\right] \exp\left[-c_2 e^{\beta t(1-\theta^2-\delta-kd)}\right],$$
which is SES in $t$ since $a<1-\theta^2-\delta-kd$ by the choice of $\delta$. Therefore, it follows from \eqref{panter} that $P(A_t\mid E_t)$ is SES in t as well. This completes the proof in view of \eqref{eq41} and \eqref{eq42}.
\end{proof}

Next, we prove Corollary~\ref{corollary1}. We prove Corollary~\ref{corollary1} before Theorem~\ref{theorem1} since the former will be used to prove the latter; nonetheless, we prefer to call the former a corollary of the latter, since the latter is a stronger result, which can be used to prove the former as well.

\subsection{Proof of Corollary~\ref{corollary1}}
Define
$$\Omega_0:=\{\omega:\forall \,\varepsilon>0\:\:\exists\:  t_0=t_0(\omega)\: \text{such that} \:\: \forall \:t\geq t_0\:\:\Bigl\lvert\frac{1}{t}\log Z_t(B_t)-\beta(1-\theta^2-kd)\Bigr\rvert<\varepsilon \}.$$
We will show that $P(\Omega_0)=1$. Let $\varepsilon>0$. For all large $t$, 
\begin{align}
& P\left(\Bigl\lvert\frac{1}{t}\log Z_t(B_t)-\beta(1-\theta^2-kd)\Bigr\rvert\geq\varepsilon\right) \nonumber \\
&\quad=P\left(\frac{1}{t}\log Z_t(B_t)\geq \beta(1-\theta^2-kd)+\varepsilon\right)+P\left(\frac{1}{t}\log Z_t(B_t)\leq \beta(1-\theta^2-kd)-\varepsilon\right) \nonumber \\
&\quad=P\left(Z_t(B_t)\geq \exp[\beta(1-\theta^2-kd)t+\varepsilon t]\right)+P\left(Z_t(B_t)\leq \exp[\beta(1-\theta^2-kd)t-\varepsilon t]\right) \label{eq47} \\
&\quad\leq\frac{E[Z_t(B_t)]}{\exp[\beta(1-\theta^2-kd)t+\varepsilon t]}+e^{-c_0 t}\leq e^{-\varepsilon t+o(t)}+e^{-c_0 t}\leq e^{-c_1 t}, \label{eq48}
\end{align}
where we have used the Markov inequality to bound the first term, and Lemma~\ref{lemma2} to bound the second term on the right-hand side of \eqref{eq47}, and used \eqref{panter2} to bound $E[Z_t(B_t)]$ in the last line. For $n\in\mathbb{N}$, define the events
$$A_n:=\left\{\biggl\lvert\frac{1}{n}\log Z_n(B_n)-\beta(1-\theta^2-kd)\biggr\rvert\geq\varepsilon\right\}.$$
Then, by \eqref{eq48}, there exists $m\in\mathbb{N}$ such that for each $n\geq m$, $P(A_n)\leq e^{-c_1 n}$. Since 
$$\sum_{n=1}^\infty P(A_n)=c_2+\sum_{n=m}^\infty P(A_n)=c_2+\sum_{n=m}^\infty e^{-c_1 n}\leq c_2+\frac{1}{1-e^{-c_1}}<\infty ,$$
by Borel-Cantelli lemma it follows that $P(A_n\:\:\text{occurs}\:\:\text{i.o.})=0$, where i.o. stands for \emph{infinitely often}. Choosing $\varepsilon=1/k$, this implies that for each $k\geq 1$, we have 
$$ P(\Omega_k)=1,\quad \Omega_k:=\left\{\omega:\exists\:  n_0=n_0(\omega)\: \text{such that} \:\: \forall \:n\geq n_0\:\:\Bigl\lvert\frac{1}{n}\log Z_n(B_n)-\beta(1-\theta^2-kd)\Bigr\rvert<\frac{1}{k}\right\}.$$
Then, $P(\Omega_0)=P\left(\bigcap_{k\geq 1}\Omega_k\right)=1$ since $P(\Omega_k)=1$ for each $k\geq 1$. \qed

\subsection{Proof of Theorem~\ref{theorem1}}

Theorem~\ref{theorem1} is proved in the same spirit as \cite[Thm.\ 1]{O2018}. For the lower bound, we find a strategy that realizes the desired event with optimal probability on a logarithmic scale. The proof of the upper bound can be viewed as the second step of a bootstrap argument, whose first step was completed by Lemma~\ref{lemma2}; it sharpens the constant on the right-hand side of \eqref{eq40} so as to show that the strategy that gives the lower bound is indeed optimal.

\subsubsection{Proof of the lower bound}

Fix $0\leq \theta<1$, $0\leq k<(1-\theta^2)/d$, and a unit vector $\textbf{e}$. For $0\leq a<1-\theta^2-kd$, define the event
$$A_t=\left\{Z_t(B_t)<e^{\beta at}\right\},$$
and define
$$\bar\rho=\bar\rho(\theta,k,a)=1-(a+kd)/2-\sqrt{((a+kd)/2)^2+\theta^2},$$
which is chosen so that $(1-\bar\rho)^2-(a+kd)(1-\bar\rho)=\theta^2$. Let $0<\rho\leq\bar\rho$ and $\varepsilon>0$. Let $E_t$ be the event that in the time interval $[0,\rho t]$, the branching is completely suppressed and the initial Brownian particle is moved to a distance of 
$$d(t):=\left(\sqrt{(1-\rho)^2-(a+kd)(1-\rho)}-\theta+\varepsilon\right)\sqrt{2\beta}t+o(t)$$
from the origin in the opposite direction of $\mathbf{e}$. By the independence of branching and motion mechanisms of BBM, this partial strategy over $[0,\rho t]$ has probability
\begin{equation}P(E_t)=\exp\left[-\beta\left(\rho+\frac{\left(\sqrt{(1-\rho)^2-(a+kd)(1-\rho)}-\theta+\varepsilon\right)^2}{\rho}\right)t+o(t)\right], \label{eq49}
\end{equation}
where the first term under the exponent comes from suppressing the branching, and the second term from the linear Brownian displacement. By the Markov property applied at time $\rho t$, it is clear that $P(A_t\mid E_t)$ is the same as the probability that a BBM starting with a single particle at position $(-d(t)+o(t))\textbf{e}$ gives a mass of less than $e^{\beta at}$ to $B_t$ at time $(1-\rho)t$. Since the distance between the position of the single particle at time $\rho t$ and the center of $B_t$ is 
$$d(t)+\theta\sqrt{2\beta}t+o(t)=\left(\sqrt{(1-\rho)^2-(a+kd)(1-\rho)}+\varepsilon\right)\sqrt{2\beta}t+o(t),$$
and since
$$(1-\rho)-\frac{\left(\sqrt{(1-\rho)^2-(a+kd)(1-\rho)}+\varepsilon\right)^2}{1-\rho}-kd<a ,$$
Corollary~\ref{corollary1} implies that $P(A_t\mid E_t)=\exp[o(t)]$. Then, from the estimate $P(A_t)\geq P(E_t)P(A_t \mid E_t)$ and \eqref{eq49}, it follows that
\begin{equation} \underset{t\rightarrow\infty}{\liminf}\:\frac{1}{t}\log P(A_t)\geq-\beta\left[\rho+\frac{\left(\sqrt{(1-\rho)^2-(a+kd)(1-\rho)}-\theta+\varepsilon\right)^2}{\rho}\right]. \label{eq50}
\end{equation}
Optimize the right-hand side of \eqref{eq50} over $\rho\in(0,\bar\rho]$ to complete the proof of the lower bound.

\subsubsection{Proof of the upper bound}

We refer the reader to the proof of the upper bound of \cite[Thm.\ 1]{O2018}; simply change the parameter $a$ by $a+kd$ in the equations (19), (21), (24)-(26) therein. The proof of the upper bound of Theorem~\ref{theorem1} is otherwise identical to that of \cite[Thm.\ 1]{O2018}. We note that in the present work a similar (but not identical) technique is used later for the proof of the upper bound of Theorem~\ref{theorem2}. Therefore, to avoid duplication, here we simply refer the reader to the proof of \cite[Thm.\ 1]{O2018}.

\section{Density of BBM}\label{section5}

In this section, we prove Theorem~\ref{theorem2}. The lower bound is a direct consequence of Theorem~\ref{theorem1}. The proof of the upper bound uses a method similar to that of \cite[Thm.1]{E2003} and \cite[Thm.1]{O2018}, along with elementary geometric arguments. Also, it can be viewed as the second step of a bootstrap argument, whose first step was completed by Lemma~\ref{lemma2}.  

\subsection{Theorem~\ref{theorem2} -- Proof of the lower bound}

Fix $0<\theta<1$, $0\leq k<(1-\theta^2)/d$ and $r_0>0$. Let $0<\theta'<\theta$. Then, $0\leq k<(1-\theta'^2)/d$ as well. For $t\geq 0$ consider the ball $B_t:=B(x_t\mathbf{e},r_t)$, where $x_t=\theta'\sqrt{2\beta}t$, $r_t=r_0\,e^{-\beta kt}$, and $\mathbf{e}=(1,0,\ldots,0)$ is the unit vector in the direction of the first coordinate. Then, by Theorem~\ref{theorem1}, 
$$P(Z_t(B_t)=0)=\exp\left[-\beta\,I(\theta',k,0)+o(t)\right].$$
Since $\{Z_t(B_t)=0\}\subseteq A_t^r= \left\{\text{supp}(Z(t))\:\: \text{is not $r_t$-dense in $B(0,\theta\sqrt{2\beta}t)$}\right\}$ for all large $t$, it follows that
$$\underset{t\rightarrow\infty}{\liminf}\:\frac{1}{t}\log P\left(A_t^r\right)\geq -\beta I(\theta',k,0).$$
Let $\theta'\to\theta$ and use the continuity of $I(\theta',k,0)$ in $\theta'$ to complete the proof.

\subsection{Theorem~\ref{theorem2} -- Proof of the upper bound}

Fix $0<\theta<1$, $0\leq k<(1-\theta^2)/d$ and $r_0>0$, and for $t\geq 0$ set $r_t=r_0\,e^{-\beta kt}$. Throughout this subsection, we use
$$B_t:=B(\theta\sqrt{2\beta}t\mathbf{e},r_0), \quad \mathcal{B}_t:=B(0,\theta\sqrt{2\beta}t)  .$$ 
The proof is broken into three parts for better readability. The first two parts are on the $r_t$-density of BBM only within $B_t$. The last part extends the $r_t$-density of BBM to the entire subcritical ball $\mathcal{B}_t$. In the rest of the proof, fix the dimension $d$, and let
\begin{equation}
n_t:=\left\lceil 2\sqrt{d}\,e^{\beta kt} \right\rceil^d. \label{nd}
\end{equation}

\subsubsection{Part I: Any $n_t$-collection of balls within $B_t$}
Let $\left(x_j:1\leq j\leq n_t\right)$ be any collection of $n_t$ points in $B_t$, where we suppress the $t$-dependence of $x_j$ for ease of notation. For each $j$, define $B_t^j:=B(x_j,r_t/(2\sqrt{d}))$ so that each $B_t^j$ is a ball with radius $r_t/(2\sqrt{d})$ and center lying in $B_t$.

We split the time interval $[0,t]$ into two pieces at $\rho t$, $\rho \in [0,1]$, which is the instant at which the total mass exceeds $\lfloor t \rfloor$. In the first piece, the branching is partially suppressed to give polynomially many particles only, which has an exponential probabilistic cost; whereas we are able to keep all of these particles close enough to the origin (at sublinear distance) at no cost since there are not exponentially many of them. In the second piece, given that we now have $\lfloor t \rfloor$ particles close enough to the origin, we argue that with overwhelming probability, for each $j$, $1\leq j\leq n_t$, there is at least one of these particles such that the sub-BBM it initiates at time $\rho t$ contributes a particle to $B_t^j$ at time $t$. To catch the optimal $\rho$, we discretize $[0,t]$ into many small pieces, and condition the process on in which piece $\rho t$ falls, which results in a sum of terms, of which only the largest contributes on a logarithmic scale. 

For $t>0$ and $1\leq j\leq n_t$, define the events
$$A_t^j:=\{Z_t(B_t^j)=0\},\quad A_t:=\bigcup_{1\leq j \leq n_t}A_t^j.$$
Observe that $A_t$ is the event that at least one $B_t^j$ is empty at time $t$.

Recall that $N_t=Z_t(\mathbb{R}^d)$, and for $t>1$ define the random variable
\begin{equation} \rho_t=\sup\left\{\rho \in [0,1]:N_{\rho t}\leq \lfloor t\rfloor \right\}. \nonumber
\end{equation}
Observe that for $x\in[0,1]$, we have $\{\rho_t\geq x \}\subseteq \{N_{xt}\leq \lfloor t\rfloor+1 \}$. We start by conditioning on $\rho_t$. Recall the definition of $\bar\rho$ from \eqref{rhobar} and set
$$\bar{\rho}:=\bar{\rho}(\theta,k,0)=1-(kd)/2-\sqrt{(kd/2)^2+\theta^2} .$$
Note that $\bar{\rho}>0$ since $kd<1-\theta^2$. Choose $n_0\in\mathbb{N}$ large enough so that $\lfloor \bar\rho n_0-1\rfloor-1\geq 0$. 
Then, for every $n\geq n_0$,
\begin{align} P(A_t)&\:=\sum_{i=0}^{\lfloor \bar\rho n-1\rfloor-1} P\left(A_t\cap \left\{\frac{i}{n}\leq \rho_t <\frac{i+1}{n}\right\} \right)+ P\left(A_t\cap \left\{\rho_t\geq \frac{\lfloor \bar\rho n-1\rfloor}{n}\right\} \right) \nonumber \\
&\:\:\leq \sum_{i=0}^{\lfloor \bar\rho n-1\rfloor-1} \exp\left[-\beta \frac{i}{n}t+o(t)\right]P^{(i,n)}_t\left(A_t\right)+\exp\left[-\beta \left(\bar\rho-\frac{2}{n}\right) t+o(t)\right], \label{eq21}
\end{align}
where we use \eqref{eqprop}, which implies $P(N_{(i/n)t}\leq \lfloor t\rfloor+1)=\exp[-\beta(i/n)t+o(t)]$, to control \newline
$P(\frac{i}{n}\leq \rho_t <\frac{i+1}{n})$, and introduce the conditional probabilities
\begin{equation} P^{(i,n)}_t(\cdot)=P\left(\ \cdot \ \middle| \ \frac{i}{n}\leq \rho_t <\frac{i+1}{n}\right), \quad i=0,1,\ldots,\lfloor \bar\rho n-1\rfloor-1. \nonumber
\end{equation}

For each pair $(i,n)$, where $n\geq n_0$ and $i=0,1,\ldots,\lfloor \bar\rho n-1\rfloor-1$, define the interval 
$$I^{(i,n)}:=[i/n,(i+1)/n),$$
and the radius
\begin{equation}
r_t^{(i,n)}:=\sqrt{2\beta}\left(\sqrt{\left(1-\frac{i+1}{n}\right)^2-kd\left(1-\frac{i+1}{n}\right)}-\theta-\varepsilon\right)t,\label{eqradius}
\end{equation}
where $\varepsilon=\varepsilon(n)>0$ is chosen small enough so that \eqref{eqradius} is positive for each $i=0,1,\ldots,\lfloor \bar\rho n-1\rfloor-1$. By definition of $\rho_t$, conditional on the event $\rho_t\in I^{(i,n)}$, there exists an instant in $[ti/n,t(i+1)/n)$, namely $\rho_t t$, at which there are exactly $\lfloor t\rfloor+1$ particles in the system. Let $E_t^{(i,n)}$ be the event that among the $\lfloor t\rfloor+1$ particles alive at $\rho_t t$, there is at least one outside $B_t^{(i,n)}:=B\left(0,r_t^{(i,n)}\right)$. Estimate
\begin{equation} P^{(i,n)}_t\left(A_t\right)\leq P^{(i,n)}_t\left(E_t^{(i,n)}\right)+P^{(i,n)}_t\left(A_t \mid [E_t^{(i,n)}]^c\right). \label{eq22}
\end{equation}
If the event $E_t^{(i,n)}$ occurs, then at least one among $\lfloor t \rfloor+1$ many particles has escaped $B_t^{(i,n)}$ by time at most $t(i+1)/n$. Note that each particle alive at time $s$ is at a random point, whose spatial distribution is identical to that of $X(s)$. Therefore, by Proposition B and the union bound, we have
\begin{equation} P^{(i,n)}_t\left(E_t^{(i,n)}\right)\leq \: (\lfloor t \rfloor+1)\,\exp\left[-\frac{\left(r_t^{(i,n)}\right)^2}{2(i+1)/n}t+o(t)\right]
= \: \exp\left[-\frac{\left(r_t^{(i,n)}\right)^2}{2(i+1)/n}t+o(t)\right] .  \label{eq23}
\end{equation} 
Now consider the second term on the right-hand side of (\ref{eq22}). Choose $\ell$ from $\{1,2,\ldots,n_t\}$ such that $|x_\ell|\geq |x_j|$ for all $1\leq j\leq n_t$. Then, letting $A_t^\ell=\{Z_t(B_t^\ell)=0\}$, it follows from $A_t=\bigcup_{1\leq j\leq n_t}A_t^j$, the union bound and Lemma~\ref{lemma1} that
\begin{equation}
P^{(i,n)}_t\left(A_t \mid [E_t^{(i,n)}]^c\right) \leq n_t\, P^{(i,n)}_t\left(A_t^\ell \mid [E_t^{(i,n)}]^c\right). \label{kedibaba2}
\end{equation}
On the event $[E_t^{(i,n)}]^c$, there are $\lfloor t\rfloor+1$ particles in $B_t^{(i,n)}$ at time $\rho_t t$. Then, conditional on $\rho_t\in I^{(i,n)}$, Lemma~\ref{lemma2} and \eqref{eqradius} imply that with overwhelming probability, the sub-BBM emanating from each such particle at time $\rho_t t$ evolves in the remaining time of length at least $(1-(i+1)/n)t$ to contribute at least one particle to $B_t^\ell$ at time $t$. This is due to
\begin{equation}
0<\left(1-\frac{i+1}{n}\right)-\frac{\left(\sqrt{\left(1-\frac{i+1}{n}\right)^2-kd\left(1-\frac{i+1}{n}\right)}-\varepsilon\right)^2}{1-\frac{i+1}{n}}-kd, \label{eq25}
\end{equation}
that is, the distance between $B_t^\ell$ and the starting point of the sub-BBM is too short for the sub-BBM to contribute no particles to $B_t^\ell$ at time $t$. More precisely, let $p_t^y$ be the probability that a BBM starting with a single particle at position $y\in\mathbb{R}^d$ contributes no particles to $B_t^\ell$ at time $t$. Then, conditional on $\rho_t\in I^{(i,n)}$ (since $\rho_t t<t(i+1)/n$), by Lemma~\ref{lemma2} and \eqref{eq25}, uniformly over $y\in B_t^{(i,n)}$, there exists $c>0$ and $t_0$ such that
$$p_{t(1-\rho_t)}^y \leq e^{-ct} \quad \text{for all} \quad t\geq t_0. $$   
Hence, by the strong Markov property applied at time $\rho_t t$ and the independence of particles present at that time, for all large $t$ we have
\begin{equation} P^{(i,n)}_t\left(A_t^\ell \mid [E_t^{(i,n)}]^c\right)\leq \left(e^{-ct}\right)^{\lfloor t\rfloor+1}\leq e^{-ct^2}, \label{kedibaba}
\end{equation}
which is SES in $t$. It follows from \eqref{kedibaba2} and \eqref{kedibaba} that
\begin{equation}
P^{(i,n)}_t\left(A_t \mid [E_t^{(i,n)}]^c\right) \leq n_t\,e^{-ct^2}=\left\lceil 2\sqrt{d}\,e^{\beta kt} \right\rceil^d \,e^{-ct^2}=e^{-ct^2+o(t^2)}. \label{kedibaba3}
\end{equation}
From \eqref{eqradius}, \eqref{eq22}, \eqref{eq23} and \eqref{kedibaba3}, we obtain
\begin{equation} P^{(i,n)}_t\left(A_t\right) \leq \exp\left[-\beta\,\frac{\left(\sqrt{\left(1-\frac{i+1}{n}\right)^2-kd\left(1-\frac{i+1}{n}\right)}-\theta-\varepsilon\right)^2}{(i+1)/n}\,t+o(t)\right]+\exp\left[-ct^2+o\left(t^2\right)\right]. \label{eq24}
\end{equation}
Substituting \eqref{eq24} into \eqref{eq21}, and optimizing over $i\in\left\{0,1,\ldots,\lfloor \bar\rho n-1\rfloor-1\right\}$ gives
\begin{align} &\underset{t\rightarrow \infty}{\limsup}\:\frac{1}{t}\log P\left(A_t\right) 
\leq \nonumber \\
& -\beta\left[\underset{i\in\left\{0,1,\ldots,\lfloor \bar\rho n-1\rfloor-1\right\}}{\min}
\left\{\frac{i}{n}+\frac{\left(\sqrt{\left(1-\frac{i+1}{n}\right)^2-kd\left(1-\frac{i+1}{n}\right)}-\theta-\varepsilon\right)^2}{(i+1)/n}\right\}\wedge\left(\bar\rho-\frac{2}{n}\right)\right], \label{big}
\end{align}
where we use $a\wedge b$ to denote the minimum of $a$ and $b$. Now first let $\varepsilon \rightarrow 0$, then set $\rho=i/n$, let $n\rightarrow \infty$, and use the continuity of the functional form from which the minimum is taken to obtain
\begin{equation} \underset{t\rightarrow\infty}{\limsup}\:\frac{1}{t}\log P\left(A_t\right)\leq -\beta\underset{\rho\in(0,\bar\rho]}{\inf}\left[\rho+\frac{\sqrt{(1-\rho)^2-kd(1-\rho)}-\theta}{\rho}\right]=-\beta I(\theta,k,0).  \label{big2}
\end{equation}
(Note that we have not written the last term on the right-hand side of \eqref{big} explicitly in \eqref{big2}, because once $n\rightarrow \infty$, this term becomes $\bar\rho$, which is attained by the function inside the infimum on the right-hand side of \eqref{big2} if we set $\rho=\bar\rho$.)

\begin{remark}  We note that applying the union bound on $P\left(\cup_{1\leq j\leq n_t}A_t^j\right)$ naively along with Theorem~\ref{theorem1} does not yield the desired upper bound. Indeed, this argument gives
\begin{equation} P(A_t)=P\left(\cup_{1\leq j\leq n_t}A_t^j\right)\leq n_t\,P(A_t^\ell)=\exp[-\beta t(I(\theta,k,0)-kd)+o(t)].
\end{equation}
\end{remark}

\subsubsection{Part II: Choosing the $n_t$-collection of balls within $B_t$}
We now choose the collection of $n_t$ points $\left(x_j:1\leq j\leq n_t\right)$ in $B_t$ in a useful way. Let $C(0,r_0)$ be the cube centered at the origin with side length $2r_0$ so that $B(0,r_0)$ is inscribed in $C(0,r_0)$. Consider the simple cubic packing of $C(0,r_0)$ with balls of radius $r_0 e^{-\beta k t}/(2\sqrt{d})$. Then, at most $n_t=\left\lceil 2\sqrt{d}\,e^{\beta kt} \right\rceil^d$ balls are needed to completely pack $C(0,r_0)$, say with centers $\left(y_j:1\leq j\leq n_t\right)$. For each $j\in\{1,2,\ldots,n_t\}$, choose $z_j=y_j$ if $y_j\in B(0,r_0)$; otherwise, let $z_j=0$. In a simple cubic packing, the distance between a point in space and its closest packing ball's farthest point (we consider the farthest point to cover the worst case scenario, corresponding to $Z$ hitting the farthest point of the closest $B_t^j$), is less than the distance between the center and any vertex of a $d$-dimensional cube with side length four times the radius of a packing ball. In this case, four times the radius of a packing ball is $2r_0 e^{-\beta k t}/\sqrt{d}$. Then, since the distance between the center and any vertex of the $d$-dimensional unit cube is $\sqrt{d}/2$, it follows that for any $x\in B(0,r_0)$, there exists $z_j$ with 
$$\underset{y\in B(z_j,r_0 e^{-\beta k t}/(2\sqrt{d}))}{\text{max}} |x-y|<r_0 e^{-\beta kt}.$$
In other words,
$$\underset{1\leq j\leq n_t}{\text{min}}\:\underset{y\in B(z_j,r_0 e^{-\beta k t}/(2\sqrt{d}))}{\text{max}} |x-y|<r_0 e^{-\beta kt}. $$
For each $1\leq j\leq n_t$, let $x_j=z_j+\theta\sqrt{2\beta}t\mathbf{e}$. Then, by translation invariance, for any $x\in B_t$,
$$\underset{1\leq j\leq n_t}{\text{min}}\:\underset{y\in B(x_j,r_0 e^{-\beta k t}/(2\sqrt{d}))}{\text{max}} |x-y|< r_0 e^{-\beta kt}=r_t.$$
Define $$\widehat{A}_t:=\left\{\text{supp}(Z(t))\:\: \text{is not $r_t$-dense in $B_t$}\right\}.$$
Then, with the choice of the collection $\left(x_j:1\leq j\leq n_t\right)$, the event $A_t$ from part I of the proof satisfies $\widehat{A}_t \subseteq A_t$, and \eqref{big2} implies that 
\begin{equation}
\underset{t\rightarrow\infty}{\limsup}\:\frac{1}{t}\log P\left(\widehat{A}_t\right)\leq -\beta I(\theta,k,0). \label{density3}
\end{equation}

\subsubsection{Part III: Extension from $B_t$ to the entire subcritical ball}
By a similar geometric argument as the one in part II of this proof, we now extend the result on the density of BBM in $B_t$ to the density in the entire subcritical ball $\mathcal{B}_t$.

\smallskip

Recall that $\rho_t:=\theta\sqrt{2\beta}t$, and define
$$m_t:=\left\lceil 2\sqrt{d}\, \rho_t \frac{1}{r_0} \right\rceil^d.$$
Let $\left(\bar{x}_j:1\leq j\leq m_t\right)$ be any collection of $m_t$ points in $\mathcal{B}_t:=B(0,\rho_t)$. For each $j$, define $\mathcal{B}^j_t:=B(\bar{x}_j,r_0)$. Next, for $t>0$ and $1\leq j\leq m_t$, define the events
$$E_t^j:=\{\text{supp}(Z(t))\:\: \text{is not $r_t$-dense in $\mathcal{B}^j_t$}\},\quad E_t:=\bigcup_{1\leq j \leq m_t}E_t^j.$$
Recall that $B_t:=B(\rho_t\mathbf{e},r_0)$, and $|\bar{x}_j|\leq \rho_t$ for all $j$. Then, using the union bound, \eqref{density3}, and Lemma~\ref{lemma1}, we can bound the probability that the BBM is not $r_t$-dense in at least one $\mathcal{B}^j_t$ as 
\begin{equation}
P(E_t)=P\left(\bigcup_{1\leq j\leq m_t}E_t^j\right)\leq m_t\,e^{-\beta I(\theta,k,0)+o(t)}=e^{-\beta I(\theta,k,0)+o(t)} \label{babayaro}
\end{equation} 
since $m_t$ is only a polynomial factor.

\smallskip 

We now choose the collection $\left(\bar{x}_j:1\leq j\leq m_t\right)$ in a useful way. Let $C(0,\rho_t)$ be the cube centered at the origin with side length $2\rho_t$. The simple cubic packing of $C(0,\rho_t)$ requires at most $m_t=\left\lceil 2\sqrt{d}\, \rho_t \frac{1}{r_0} \right\rceil^d$ balls of radius $r_0/(2\sqrt{d})$, say with centers $(\bar{y}_j:1\leq j\leq m_t)$. For each $1\leq j\leq m_t$, choose $\bar{x}_j=\bar{y}_j$ if $\bar{y}_j\in B(0,\rho_t)$; otherwise, let $\bar{x}_j=0$. Then, by an argument similar to the one in part II of this proof, one can show that for any $x\in B(0,\rho_t)$,
$$\underset{1\leq j\leq m_t}{\text{min}} |x-\bar{x}_j|<r_0, $$
which implies that
$$B(0,\rho_t)\subseteq \bigcup_{1\leq j\leq m_t} B(\bar{x}_j,r_0).$$
In other words, we are enlarging the packing ball radius from $r_0/(2\sqrt{d})$ to $r_0$ so that every point in $B(0,\rho_t)$ falls inside at least one enlarged ball $B(\bar{x}_j,r_0)$. Then, with the choice of the collection $(\bar{x}_j:1\leq j \leq m_t)$, the event $E_t$ from above satisfies 
$$A_t^r:=\left\{\text{supp}(Z(t))\:\: \text{is not $r_t$-dense in $B(0,\rho_t)$}\right\}\subseteq \bigcup_{1\leq j \leq m_t}E_t^j=E_t,$$
and \eqref{babayaro} implies that
\begin{equation} \underset{t\rightarrow\infty}{\limsup}\:\frac{1}{t}\log P\left(A_t^r\right)\leq -\beta I(\theta,k,0). \nonumber
\end{equation} 
This completes the proof of the upper bound of Theorem~\ref{theorem2}.

\section{Enlargement of BBM}\label{section6}

For a BBM $Z=(Z(t))_{t\geq 0}$, recall the definition of its \emph{$r$-enlargement at time $t$} as
$$Z_t^r:=\bigcup_{x\,\in\,\text{supp}(Z(t))}B(x,r).$$

\subsection{Proof of Corollary~\ref{corollary3}}
Fix $0<\theta<1$, $0\leq k<(1-\theta^2)/d$ and $r_0>0$, and for $t\geq 0$ let $\rho_t=\theta\sqrt{2\beta}t$ and $r_t=r_0\,e^{-\beta kt}$. Observe the equality of events 
$$\{B(0,\rho_t)\subseteq Z_t^{r_t}\}=\left\{\text{supp}(Z(t))\:\: \text{is $r_t$-dense in $B(0,\rho_t)$}\right\}=\left(A_t^r\right)^c.$$
Then, Corollary~\ref{corollary3} can be proved by using \eqref{thm2} in Theorem~\ref{theorem2} via a standard Borel-Cantelli argument similar to the one in the proof of Corollary~\ref{corollary1}. To avoid repetition, we omit the details.

\subsection{Proof of Theorem~\ref{theorem3}}
We will show that for every $\varepsilon>0$ there exist positive constants $c_1$ and $c_2$ such that for all large $t$, 
\begin{equation} 
P\left(\frac{\textsf{vol}\left(Z_t^{r_t}\right)}{t^d}\leq [2\beta(1-kd-\varepsilon)]^{d/2}\omega_d\right)\leq e^{-c_1 t} ,\label{eq301}
\end{equation}
and
\begin{equation} 
P\left(\frac{\textsf{vol}\left(Z_t^{r_t}\right)}{t^d}\geq [2\beta(1-kd+\varepsilon)]^{d/2}\omega_d\right)\leq e^{-c_2 t}.\label{eq302}
\end{equation}
Then, Theorem~\ref{theorem3} will follow from \eqref{eq301} and \eqref{eq302} via a Borel-Cantelli argument similar to the one in the proof of Corollary~\ref{corollary1}. 

Let $\varepsilon>0$ and set $\theta=\theta_1=\sqrt{1-kd-\varepsilon/2}$ in Theorem~\ref{theorem2}, which gives $\rho_t=\theta_1\sqrt{2\beta}t=\sqrt{2\beta(1-kd-\varepsilon/2)}t$. Then, $0\leq k <(1-\theta_1^2)/d=k+\varepsilon/(2d)$ so that Theorem~\ref{theorem2} applies, and gives
$$P(A_t^r)=\exp[-\beta\,I(\theta_1,k,0)t+o(t)] .$$
This proves \eqref{eq301} since for all large $t$, $\left\{\textsf{vol}\left(Z_t^{r_t}\right)/t^d\leq [2\beta(1-kd-\varepsilon)]^{d/2}\omega_d \right\}\subseteq A_t^r$. 

\smallskip

To prove \eqref{eq302}, for $\theta\geq 0$ and $t\geq 0$, let $\mathcal{N}_t^\theta$ be the set of particles outside $B(0,\theta\sqrt{2\beta}t)$ at time $t$. Set $\theta_2=\sqrt{1-kd+\varepsilon/2}$. Then, 
$$E(|\mathcal{N}_t^{\theta_2}|)=\exp[\beta t(1-\theta_2^2)+o(t)]=\exp[\beta t(kd-\varepsilon/2)+o(t)],$$ and the Markov inequality yields
\begin{equation}
P\left(|\mathcal{N}_t^{\theta_2}|\geq\exp[\beta t(kd-\varepsilon/4)]\right)\leq \exp[-\beta t\varepsilon/4+o(t)]. \label{eq304}
\end{equation} 
For an upper bound on $\textsf{vol}\left(Z^{r_t}_t\right)$, suppose that $B(0,\theta_2\sqrt{2\beta}t)\subseteq Z^{r_t}_t$ and that the balls of radius $r_t$ centered at the positions of the particles at time $t$ are all disjoint from one another. Since the volume of a ball of radius $r_t$ is $\omega_d (r_0e^{-\beta kt})^d$, it then follows from \eqref{eq304} that
\begin{equation} P\left(\textsf{vol}\left(Z^{r_t}_t\cap (B(0,\theta_2\sqrt{2\beta}))^c\right)\geq \omega_d r_0^d \exp[-\beta t \varepsilon/4]\right)\leq \exp[-\beta t\varepsilon/4+o(t)], \nonumber
\end{equation}
where we use $A^c$ to denote the complement of a set $A$ in $\mathbb{R}^d$. This implies \eqref{eq302}, and completes the proof of Theorem~\ref{theorem3}.

\bibliographystyle{plain}

\end{document}